\documentclass[11pt]{amsart}
\usepackage{amssymb}
\usepackage{times}
\newtheorem{theorem}{Theorem}[section]
\newtheorem{lemma}[theorem]{Lemma}

\newtheorem{corollary}[theorem]{Corollary}
\theoremstyle{definition}

\theoremstyle{remark}
\newtheorem{remark}[theorem]{Remark}
\begin{document}

\title{Deciding finiteness of matrix groups in positive characteristic}

\author{A. S. Detinko \and D. L. Flannery \and E. A. O'Brien}



\maketitle

\begin{abstract}
We present a new algorithm to decide finiteness of matrix groups
defined over a field of positive characteristic. Together with
previous work for groups in zero characteristic, this provides the
first complete solution of the finiteness problem for finitely
generated matrix groups over an arbitrary field. We also give an
algorithm to compute the order of a finite matrix group over a
function field of positive characteristic. Our implementations of
these algorithms are publicly available in {\sc Magma}.
\end{abstract}

\section{Introduction}

Deciding finiteness is a fundamental problem for any class of
potentially infinite groups. For matrix groups over a field of
zero characteristic, the algorithms of \cite{bbr,JSC4534} provide
a solution of this problem, and their implementations perform
satisfactorily for reasonably large input (\mbox{cf.}
\cite[Section 4]{JSC4534}). Deciding finiteness over a purely
transcendental extension $\mathbb{F}$ of a finite field was
considered by several authors \cite{Detinko01,ivanyos01,rtb99}.
The approach taken in \cite{rtb99} relies on the fact that a
subgroup $G$ of $\mathrm{GL}(n,\mathbb{F})$ is finite if and only
if, for every finite subfield $\mathbb{F}_q$ of $\mathbb{F}$, the
enveloping algebra $\langle G \rangle_{\mathbb{F}_q}$ is finite.
Since the dimension $\mathrm{dim}_{\mathbb{F}_q}\langle G
\rangle_{\mathbb{F}_q}$ of this algebra may depend exponentially
on $n$ (see \cite[Theorem 3.3]{rtb99}), this leads to
exponential-time algorithms. The polynomial-time algorithms of
\cite{Detinko01,ivanyos01} involve significant computing over
function fields, and so we expect that they are practical only for
small input. We know of no implementations of the algorithms of
\cite{Detinko01,ivanyos01,rtb99}.

A uniform approach to deciding finiteness of matrix groups over an
infinite field via congruence homomorphisms was proposed in
\cite[Section 4.3]{Draft}, and applied to nilpotent groups. We
implemented this approach, for rational nilpotent groups, in the
computer algebra systems {\sc Magma} \cite{Magma} and {\sf GAP}
(see the `Nilmat' package \cite{nilmat}). Its performance is
usually much better than existing procedures in {\sf GAP} and {\sc
Magma}.

The idea of using congruence homomorphisms to decide finiteness of
matrix groups was further developed in \cite{JSC4534}, for groups
over a function field of zero characteristic. In this paper we
extend the ideas of \cite{JSC4534} to positive characteristic. As
in that earlier paper, our main method is the application of
congruence homomorphisms to enable a comparison of dimensions of
certain enveloping algebras. However, the finiteness problem in
positive characteristic is more complicated: a finite subgroup of
$\mathrm{GL}(n,\mathbb{F})$ need not be completely reducible, and
it can be unboundedly large. The opposite holds in characteristic
zero.

Despite these difficulties, we obtain a substantial improvement
upon the algorithms of \cite{Detinko01,ivanyos01,rtb99}. We avoid
their most inefficient step: computing a basis of the enveloping
algebra of the input group over a function field (see
Sections~\ref{prelandbackground} and \ref{algdecfinpos}). As in
\cite{JSC4534}, much of the computation takes place in the
coefficient field---which is finite here. Although the number of
(function and finite) field operations of our finiteness testing
algorithm is polynomial in certain parameters of the input, our
primary goal was to develop a {\it practical} algorithm. We have
implemented it in {\sc Magma} \cite{Magma} and demonstrate that it
performs well for a range of input.

We also give an algorithm to compute the order of a finite matrix
group $G$ over a function field of positive characteristic, based
on the same strategy used to decide finiteness. This algorithm
finds an isomorphic copy of $G$ over a finite field, which can be
used to derive additional information about $G$. In
Section~\ref{simplenilpotent} we present a simplified finiteness
test for nilpotent groups. Finally, in Section~\ref{experiment},
we report on the performance of our {\sc Magma} implementation of
these algorithms.

By elementary structure theory of finitely generated field
extensions, any finitely generated matrix group $G$ is defined
over a finite extension of a function field. As explained below,
we can construct an isomorphism of $G$ onto a group defined over
the function field, in larger degree. Thus the results of this
paper together with \cite{bbr, JSC4534} effectively allow us to
decide finiteness of a finitely generated matrix group over any
field (\mbox{cf.} also \cite[Section 3.2.2]{JSC4534}).

\section{Preliminaries and background}
\label{prelandbackground}

Let $\mathbb{F}$ be a field of characteristic $p>0$, and let
$G=\langle \mathcal{S} \rangle$, where $\mathcal{S} = \{ S_1,
\ldots , S_r\} \subseteq$ $\mathrm{GL}(n,\mathbb{F})$. We may
assume that $\mathbb{F}$ is a finite extension of a function field
$\mathbb{E}= \mathbb{F}_q(X_1,\ldots , X_m)$, where the $X_i$ are
algebraically independent indeterminates, and $\mathbb{F}_q$ is
the finite field of size $q$. Replacement of elements of
$\mathbb{F}$ by  matrices over $\mathbb{E}$ according to the
multiplication action of $\mathbb{F}$ on an $\mathbb{E}$-basis of
$\mathbb{F}$ defines an isomorphism of $G$ into $\mathrm{GL}(n
l,\mathbb{E})$, where $l=|\mathbb{F} :\mathbb{E}|$. So without
loss of generality, from now on $\mathbb{F} = \mathbb{F}_q (X_1,
\ldots, X_m)$, $m\geq 1$, and $q$ is a power of the prime $p$.

In fact $G$ is contained in $\mathrm{GL}(n,R)$ for a finitely
generated integral domain $R \subseteq \mathbb{F}$. We can take $R
= \frac{1}{f}\, \mathbb{F}_q[X_1, \ldots , X_m]$, where $f =
f(X_1, \ldots , X_m)$ is a common multiple of the denominators of
the non-zero entries of the $S_i$ and $S_{i}^{-1}$, $1\leq i \leq
r$. We say that $\alpha = (\alpha_1, \ldots , \alpha_m)$ is {\em
admissible} (or $\mathcal{S}$-{\em admissible}) if $f(\alpha) \neq
0$. Here the $\alpha_i$ are in the algebraic closure
$\overline{\mathbb{F}}_q$ of $\mathbb{F}_q$; note that
$\mathbb{F}_q$ need not contain $\alpha_i$ such that $\alpha$ is
admissible. For an admissible $\alpha$, let $\nu$ denote the
positive integer such that $\mathbb{F}_q(\alpha):=
\mathbb{F}_q(\alpha_1, \ldots , \alpha_m) = \mathbb{F}_{q^\nu}$.
Let $\varphi_\alpha$ be the ring homomorphism $R\rightarrow
\mathbb{F}_{q^\nu}$ whose kernel is generated by the monomials
$X_i-\alpha_i$, $1\leq i \leq m$. If necessary, we extend
$\varphi_\alpha$ to $\widehat{R} = \frac{1}{f}\,
\mathbb{F}_{q^\mu}[X_1, \ldots , X_m]$ for any $\mu \geq 1$ in the
obvious way. With a slight abuse of notation, the induced
congruence homomorphisms on $\mathrm{GL}(n,\widehat{R})$ and on
the full matrix algebra $\mathrm{Mat}(n,\widehat{R})$ will also be
denoted $\varphi_\alpha$. Evaluation of $\varphi_\alpha$ on a
subset $\mathcal{M}$ of $\mathrm{Mat}(n,\widehat{R})$ is simply
substitution of $\alpha_i$ for $X_i$ in the entries of the
elements of $\mathcal{M}$, $1\leq i \leq m$. We denote
$\varphi_\alpha(\mathcal{M})$ as $\mathcal{M}(\alpha)$.
\begin{lemma}
If $G$ is finite then the kernel of $\varphi_\alpha$ on $G$ is a
$p$-group.
\end{lemma}
\begin{proof}
This holds for $m=1$ by \cite[Proposition 3.2 and Example
3.6]{Draft}. The result for $m>1$ follows readily: the kernel of a
composite of congruence homomorphisms, all of whose kernels are
$p$-groups, is a $p$-group.
\end{proof}
\begin{corollary}\label{weaker}
If $G$ is finite and completely reducible, then $\varphi_\alpha$
is an isomorphism from $G$ onto $\varphi_\alpha(G)$ for every
admissible $\alpha$.
\end{corollary}

Let $\mathbb{L}/\mathbb{K}$ be a field extension, and suppose that
$\mathcal{T}$ is a finite subset of $\mathrm{GL}(n,\mathbb{L})$
such that the enveloping algebra $\langle \mathcal{T}
\rangle_\mathbb{K}$ is finite-dimensional as a $\mathbb{K}$-vector
space. We now describe a standard procedure that constructs a
basis of $\langle \mathcal{T} \rangle_\mathbb{K}$ consisting of
elements from the monoid generated by $\mathcal{T}$. (Since we use
the procedure to compute an enveloping algebra basis only over a
finite field, we assume that $\mathbb{L}$ is finite in the
description.)

\bigskip

\hspace*{-1.5mm} ${\tt BasisEnvAlgebra}(\mathcal{T}, \mathbb{K})$

\vspace*{1mm}

Input: $\mathcal{T} \subseteq \mathrm{GL}(n, \mathbb{L})$,
$\mathbb{L}$ a finite field, and $\mathbb{K}$ a subfield of
$\mathbb{L}$.

\vspace*{1mm}

Output: a basis of the enveloping algebra $\langle H
\rangle_\mathbb{K}$, where $H = \langle \mathcal{T}\rangle$.

\vspace*{1mm}

\begin{itemize}

\item[(I)] $\mathcal{A} := \{ I_n \}$.

\vspace*{1mm}

\item[(II)]  While there exist $A \in \mathcal{A}$ and $T \in
\mathcal{T}$ such that $A T \not \in \mathrm{span}_\mathbb{K}
(\mathcal{A})$ do
 $\mathcal{A}:= \mathcal{A} \cup \{ AT\}$.

 \vspace*{1mm}

\item[(III)] Return $\mathcal{A}$.
\end{itemize}

\bigskip

We now set up a convention. Suppose that $\mathcal{S}(\alpha)$
is duplicate-free. For $A(\alpha)\in \mathrm{Mat}(n,
\mathbb{F}_{q^\nu} )$ that is a word in the elements of
$\mathcal{S}(\alpha)$, we canonically define a pre-image $A$ of
$A(\alpha)$ in $\mathrm{GL}(n,\mathbb{F})$: if $A(\alpha) =
S_{i_1}(\alpha) \cdots S_{i_t}(\alpha)$ then $A= S_{i_1} \cdots
S_{i_t}$.

\begin{lemma}\label{lindepoverextend}
$B_1, \ldots , B_l \in \mathrm{Mat}(n,\mathbb{F})$ are
$\mathbb{F}_q$-linearly independent if and only if they are
$\mathbb{F}_{q^\mu}$-linearly independent.
\end{lemma}
\begin{proof}
The non-trivial $\mathbb{F}_{q^\mu}$-linear dependence
$\sum_{i=1}^l a_i B_i = 0_n$  between the $B_i$ yields a system of
equations with coefficients in $\mathbb{F}$. Since $(a_1, \ldots ,
a_l)$ is a solution of this system, $a_i \in \mathbb{F} \cap
\mathbb{F}_{q^\mu} = \mathbb{F}_{q}$ for all $i$. Thus, if the
$B_i$ are $\mathbb{F}_q$-linearly independent then they must be
$\mathbb{F}_{q^\mu}$-linearly independent. The other direction is
obvious.
\end{proof}
\begin{corollary}\label{sameoverfqandover}
If $G$ is finite, then $\mathrm{dim}_{\mathbb{F}_q} \langle G
\rangle_{\mathbb{F}_q} = \mathrm{dim}_{\mathbb{F}_{q^\mu}} \langle
G \rangle_{\mathbb{F}_{q^\mu}}$.
\end{corollary}
\begin{proof}
By Lemma~\ref{lindepoverextend}, $\mathrm{dim}_{\mathbb{F}_q}
\langle G \rangle_{\mathbb{F}_q} \leq
\mathrm{dim}_{\mathbb{F}_{q^\mu}} \langle G
\rangle_{\mathbb{F}_{q^\mu}}$. Conversely, $\langle G
\rangle_{\mathbb{F}_{q^\mu}}$ has a basis consisting of elements
of $G$; that basis is therefore an $\mathbb{F}_{q}$-linearly
independent subset of $ \langle G \rangle_{\mathbb{F}_q}$. Hence
$\mathrm{dim}_{\mathbb{F}_{q^\mu}} \langle G
\rangle_{\mathbb{F}_{q^\mu}}\leq$ $\mathrm{dim}_{\mathbb{F}_q}
\langle G \rangle_{\mathbb{F}_q}$.
\end{proof}

We write $\widehat{\mathbb{F}}$ for
$\mathbb{F}_{q^\mu}(X_1,\ldots,X_m)$.
\begin{lemma}\label{kernelinradical}
If $G$ is finite then the kernel of $\varphi_\alpha$ on $\langle G
\rangle_{\mathbb{F}_{q^\mu}}$ is contained in the radical of
$\langle G \rangle_{\mathbb{F}_{q^\mu}}$ and the radical of
$\langle G \rangle_{\widehat{\mathbb{F}}}$.
\end{lemma}
\begin{proof}
The proofs of Proposition 3.2 and Corollary 3.3 in
\cite{Detinko01} carry over.
\end{proof}

\begin{lemma}\label{stronger}
If $G$ is completely reducible, then $G$ is finite if and only if
$\varphi_\alpha: \langle G \rangle_{\mathbb{F}_{q^\mu}}
\rightarrow \langle G(\alpha) \rangle_{\mathbb{F}_{q^\mu}}$ is an
isomorphism, for any $\mathcal{S}$-admissible $\alpha$ and $\mu
\geq 1$.
\end{lemma}
\begin{proof}
If $G$ is finite then $G$ is completely reducible over the
extension field $\widehat{\mathbb{F}}$ of $\mathbb{F}$ (see
\mbox{e.g.} \cite[1.8, \mbox{p.} 12]{Huppert}), so the radical of
$\langle G \rangle_{\widehat{\mathbb{F}}}$ is zero.
Lemma~\ref{kernelinradical} now implies that $\ker \varphi_\alpha$
on $\langle G \rangle_{\mathbb{F}_{q^\mu}}$ is trivial.
\end{proof}
Note that Lemma~\ref{stronger} implies Corollary~\ref{weaker}.

\begin{lemma}\label{obviouslinalg}
The algebras $\langle G \rangle_{\mathbb{F}_{q^\mu}}$ and $\langle
G(\alpha) \rangle_{\mathbb{F}_{q^\mu}}$ are isomorphic if and only
if $$\mathrm{dim}_{\mathbb{F}_{q^\mu}} \langle
G\rangle_{\mathbb{F}_{q^\mu}} = \mathrm{dim}_{\mathbb{F}_{q^\mu}}
\langle G(\alpha) \rangle_{\mathbb{F}_{q^\mu}}.$$
\end{lemma}
\begin{proof}
A basis of $\langle G \rangle_{\mathbb{F}_{q^\mu}}$ maps under
$\varphi_\alpha$ to a spanning set of $\langle G (\alpha)
\rangle_{\mathbb{F}_{q^\mu}}$, which is a basis if and only if the
$\mathbb{F}_{q^\mu}$-dimensions of these two algebras are equal.
\end{proof}

\begin{corollary}\label{strongercor}
If $G$ is completely reducible, then $G$ is finite if and only if,
for every $\mathcal{S}$-admissible $\alpha$,
\[
\mathrm{dim}_{\mathbb{F}_{q^\mu}} \langle G
\rangle_{\mathbb{F}_{q^\mu}} =\mathrm{dim}_{\mathbb{F}_{q^\mu}}
\langle G(\alpha) \rangle_{\mathbb{F}_{q^\mu}} =
\mathrm{dim}_{\mathbb{F}_q} \langle G \rangle_{\mathbb{F}_q}
=\mathrm{dim}_{\mathbb{F}_q} \langle G(\alpha)
\rangle_{\mathbb{F}_q}.
\]
\end{corollary}
\begin{proof}
This follows from Corollary~\ref{sameoverfqandover},
Lemma~\ref{stronger} and Lemma~\ref{obviouslinalg}.
\end{proof}

\begin{lemma}
\label{twosix} If $A_1(\alpha) , \ldots , A_d(\alpha)$ are
$\mathbb{F}_{q^\mu}$-linearly independent, then $A_1 , \ldots ,
A_d$ are $\mathbb{F}_{q^\mu}$-linearly independent.
\end{lemma}
\begin{proof}
Clear, since $\varphi_\alpha$ is $\mathbb{F}_{q^\mu}$-linear.
\end{proof}

Now we state an algorithm to decide whether an enveloping algebra
$\langle G\rangle_{\mathbb{F}_{q^\mu}}$ and its congruence image
$\langle G(\alpha) \rangle_{\mathbb{F}_{q^\mu}}$ are isomorphic,
for admissible $\alpha$ and $\mu \geq 1$. This uses the same
approach as the algorithm ${\tt IsFiniteMatGroupFuncNF}$ of
\cite{JSC4534}.

\bigskip

\hspace*{-1.5mm} ${\tt
IsIsomorphismEnvAlgebras}(\mathcal{S},\alpha, \mu)$

\vspace*{1mm}

Input: a finite subset $\mathcal{S} = \{ S_1, \ldots , S_r\}$ of
$\mathrm{GL}(n, \mathbb{F})$, an $\mathcal{S}$-admissible
$\alpha$, a positive integer $\mu$.

\vspace*{1mm}

Output: `true' if $\varphi_\alpha$ acts on $\langle
G\rangle_{\mathbb{F}_{q^\mu}}$ as an isomorphism, where $G =
\langle \mathcal{S}\rangle$; `false' otherwise.

\vspace*{1mm}

\begin{itemize}

\item[(I)] If $\mathcal{S}(\alpha)$ has duplicates then return
`false'.

\vspace*{1mm}

\item[(II)]  Construct $\mathcal{A}(\alpha) = \{ A_1(\alpha) ,
\ldots , A_d(\alpha) \}:= {\tt
BasisEnvAlgebra}(\mathcal{S}(\alpha), \mathbb{F}_{q^\mu})$.

\vspace*{.3mm}

\noindent Let $\mathcal{A}$ be the set of canonical pre-images $\{
A_1, \ldots , A_d\}$.

\vspace*{1mm}

\item[(III)] For $A_i(\alpha)\in \mathcal{A}(\alpha)$ and
$S_j(\alpha)\in \mathcal{S}(\alpha)$

\hspace*{0.3cm} find $a_k\in \mathbb{F}_{q^\mu}$ such that
$A_i(\alpha) S_j(\alpha) = \sum_{k=1}^d a_k A_k(\alpha)$.

\vspace*{.3mm}

\hspace*{0.3cm} If $A_iS_j \neq \sum_{k=1}^d a_k A_k$, then return
`false'.

\vspace*{1mm}

\item[(IV)]  Return `true'.
\end{itemize}

\bigskip

If ${\tt IsIsomorphismEnvAlgebras}(\mathcal{S},\alpha, \mu)$
returns `true' then $G$ is finite, and the set $\mathcal{A}$ found
in step (II) is a basis of $\langle
\mathcal{S}\rangle_{\mathbb{F}_{q^\mu}}=\langle
G\rangle_{\mathbb{F}_{q^\mu}}$. (For $\mathcal{A}$ is a spanning
set by step (III), and it is linearly independent by
Lemma~\ref{twosix}.) Observe that we obtain this basis after a
calculation over a finite field, rather than over the function
field $\mathbb{F}$.

By Lemma~\ref{stronger}, the following algorithm decides
finiteness of a completely reducible subgroup of
$\mathrm{GL}(n,\mathbb{F})$.

\bigskip

\hspace*{-1.5mm} ${\tt IsFiniteCRMatGroupFuncFF}(\mathcal{S})$

\vspace*{1mm}

Input: a finite subset $\mathcal{S}$ of
$\mathrm{GL}(n,\mathbb{F})$ such that $G = \langle \mathcal{S}
\rangle$ is completely reducible.

\vspace*{1mm}

Output: `true' if $G$ is finite; `false' otherwise.

\vspace*{1mm}

\begin{itemize}

\item[(I)] Find an $\mathcal{S}$-admissible $\alpha$.

\vspace*{1mm}

\item[(II)] Return ${\tt
IsIsomorphismEnvAlgebras}(\mathcal{S},\alpha, \nu)$.
\end{itemize}

\bigskip

Corollary~\ref{strongercor} implies that we can also decide
finiteness of a completely reducible group $G$ by testing whether
$\varphi_\alpha$ acts as an isomorphism on $\langle
G\rangle_{\mathbb{F}_{q^\mu}}$, for given $\mu \geq 1$. However
$\mathrm{dim}_{\mathbb{F}_{q^\mu}} \langle G(\alpha)
\rangle_{\mathbb{F}_{q^\mu}}$ might be larger than
$\mathrm{dim}_{\mathbb{F}_{q^\nu}} \langle G(\alpha)
\rangle_{\mathbb{F}_{q^\nu}}$, which is bounded above by $n^2$.

Now suppose that $G$ is a (finitely generated, not necessarily
completely reducible) subgroup of $\mathrm{GL}(n,\mathbb{F})$, and
that we know $\alpha$ where $\varphi_\alpha$ is an isomorphism on
$\langle G\rangle_{\mathbb{F}_{q^\nu}}$ if $G$ is finite. We may
now decide finiteness of $G$ just as in ${\tt
IsFiniteCRMatGroupFuncFF}$, namely, by applying ${\tt
IsIsomorphismEnvAlgebras}$. Unfortunately, such $\alpha$ need not
exist. On the other hand, there always exist $\alpha$ such that
$\varphi_\alpha$ is an isomorphism on $\langle
G\rangle_{\mathbb{F}_q}$ if $G$ is finite. We consider these
issues again at the end of Section~\ref{algdecfinpos}.

\section{Deciding finiteness and computing orders in positive characteristic}
\label{algdecfinpos}

We now present a general algorithm to decide finiteness of a
finitely generated subgroup $G$ of $\mathrm{GL}(n,\mathbb{F})$.
The approach is similar to the finiteness testing algorithm of
\cite{Detinko01}, but avoids its most complicated step: computing
a basis of $\langle G \rangle_\mathbb{F}$ over $\mathbb{F}$. We
also outline a simple method to determine the order of a finite
subgroup of $\mathrm{GL}(n,\mathbb{F})$.

We continue with established notation: $\alpha$ is an
$\mathcal{S}$-admissible $m$-tuple of elements from
$\overline{\mathbb{F}}_q$ such that $\mathcal{S}(\alpha)$ is
duplicate-free, and $\mathcal{A}(\alpha)=\{ A_1(\alpha) , \ldots,
A_d(\alpha)\}$ is a basis of $\langle G(\alpha)
\rangle_{\mathbb{F}_q(\alpha)}$ computed via ${\tt
BasisEnvAlgebra}$, with canonical pre-image $\mathcal{A}=\{ A_1 ,
\ldots, A_d\}$. For $i$ and $j$ such that $A_i(\alpha)S_j(\alpha)
=\sum_{k=1}^d a_kA_k(\alpha)$, where $a_k \in \mathbb{F}_{q^\nu} =
\mathbb{F}_q(\alpha)$, define $D =$ $A_iS_j -\sum_{k=1}^d a_kA_k$.
We assume that $p$ does not divide $\nu$. For $a\in
\mathbb{F}_{q^\nu}$, denote the trace of $a$ over $\mathbb{F}_q$
by $\mathrm{tr}(a)$:
\[
\mathrm{tr}(a) = a + \sigma (a) + \cdots + \sigma^{\nu-1}(a),
\qquad \mathrm{Gal}(\mathbb{F}_{q^\nu}/\mathbb{F}_q) = \langle
\sigma \rangle.
\]
Observe that $D' := \nu A_iS_j -\sum_{k=1}^d \mathrm{tr}(a_k)A_k$
is in $\langle G \rangle_{\mathbb{F}}$.
\begin{lemma}\label{jdashinrad}
Let $D$ and $D'$ be as defined above. If $G$ is finite and $D\neq
0_n$, then $D'$ is a non-zero element of the radical $\Re$ of
$\langle G \rangle_\mathbb{F}$.
\end{lemma}
\begin{proof}
If $D' = 0_n$ then $A_iS_j = \sum_{k=1}^d b_k A_k$ where $b_k =
\frac{1}{\nu} \mathrm{tr}(a_k) \in \mathbb{F}_q$. In fact
$A_i(\alpha)S_j(\alpha) =$ $\sum_{k=1}^d b_k A_k(\alpha)$ implies
that $b_k=a_k$ for all $k$. But this contradicts $D = A_iS_j
-\sum_{k=1}^d a_kA_k \neq 0_n$. Hence $D'$ is non-zero. We verify
that $D'\in \Re$ as in the proof of \cite[Corollary
3.5]{Detinko01}.
\end{proof}

\begin{lemma}\label{nullspaceisgmodule}
The nullspace of the radical $\Re$ of $\langle G
\rangle_\mathbb{F}$ is a non-zero $G$-module.
\end{lemma}
\begin{proof}
For all $g\in G$ and $u$ in the nullspace $U$ of $\Re$, we have
$\Re gu= \Re u=0$, since $\Re$ is an ideal of $\langle G
\rangle_\mathbb{F}$. Thus $G U \subseteq U$ as required.
\end{proof}

So if $G$ is finite and $D \neq 0_n$, then the nullspace of $D'$
contains a non-trivial $G$-module. We compute such a submodule
using the following procedure.

\bigskip

\hspace*{-1.5mm} ${\tt ModuleViaNullspace}(\mathcal{S}, E)$

\vspace*{1mm}

Input:  a finite subset $\mathcal{S}$ of
$\mathrm{GL}(n,\mathbb{F})$, and $E\in
\mathrm{Mat}(n,\mathbb{F})$.

\vspace*{1mm}

Output: a $G$-module $U$ in the nullspace of $E$, for $G=\langle
\mathcal{S}\rangle$.

\vspace*{1mm}

\begin{itemize}

\item[(I)] $U:= \mathrm{Nullspace}(E)$.

\vspace*{1mm}

\item[(II)] While there exists $S_i \in \mathcal{S}$ such that
$U\cap S_iU \neq U$ do $U:= U\cap S_iU$.

\vspace*{1mm}

\item[(III)] Return $U$.
\end{itemize}

\bigskip

Since each pass through the while loop reduces the dimension of
$U$, ${\tt ModuleViaNullspace}$  terminates in at most $n$
iterations. If $E$ is a non-zero element of $\Re$ (for example, if
$G$ is finite and $E = D'$ for $D\neq 0_n$), then the output is a
proper non-zero $G$-submodule of the underlying space $V$.

Now we present our main algorithm for deciding finiteness. We use
the following notation. Let $U$ be a $G$-submodule of $V$ and
extend a basis of $U$ to a basis of $V$. Write $G$ with respect to
the latter basis in block triangular form; $\rho_U$ denotes the
projection homomorphism from $G$ onto the block diagonal group,
whose kernel is the unitriangular subgroup that fixes $U$ and
$V/U$ elementwise.

\bigskip

\hspace*{-1.5mm} ${\tt IsFiniteMatGroupFuncFF}(\mathcal{S})$

\vspace*{1mm}

Input: a finite subset $\mathcal{S}$ of $\mathrm{GL}(n,
\mathbb{F})$.

\vspace*{1mm}

Output: `true' if $G= \langle \mathcal{S} \rangle$ is finite;
`false' otherwise.

\vspace*{1mm}

\begin{itemize}

\item[(I)] Find an $\mathcal{S}$-admissible $\alpha$ such that $p$
does not divide $\nu = |\mathbb{F}_q(\alpha)/\mathbb{F}_q|$.

\vspace*{.3mm}

\noindent If $S_i(\alpha) = S_j(\alpha)$ for distinct $S_i, S_j
\in \mathcal{S}$, then set $E = S_i-S_j$ and go to (IV).

\vspace*{1mm}

\item[(II)]  $\mathcal{A}(\alpha) := {\tt
BasisEnvAlgebra}(\mathcal{S}(\alpha), \mathbb{F}_{q^\nu}) = \{
A_1(\alpha) , \ldots , A_d(\alpha) \}$.

\vspace*{.3mm}

\noindent  Let $\mathcal{A}$ be the canonical pre-image $\{ A_1,
\ldots , A_d\}$ of $\mathcal{A}(\alpha)$.

\vspace*{1mm}

\item[(III)]
 If there exist $A_i \in \mathcal{A}$ and $S_j\in
\mathcal{S}$ such that $A_i S_j \neq \sum_{k=1}^d a_k A_k$, where
$a_k\in \mathbb{F}_{q^\nu}$ and $A_i(\alpha) S(\alpha) =$
$\sum_{k=1}^d a_k A_k(\alpha)$, then set $E= \nu A_iS_j
-\sum_{k=1}^d \mathrm{tr}(a_k)A_k$;

\vspace*{.3mm}

\noindent else return `true'.

\vspace*{1mm}

\item[(IV)]  $U_1:= {\tt ModuleViaNullspace}(\mathcal{S}, E)$.

\vspace*{.3mm}

\noindent If $U_1 =\{ 0\}$ then return `false';

\vspace*{.3mm}

\noindent else let $\rho= \rho_{U_1}$, $U_2=V/U_1$, \\
\hspace*{0.5cm}
for $k=1,2$ do \\
\hspace*{1.0cm} $\mathcal{A} := \{ \rho(A_1)\! \mid_{U_k} , \ldots
, \rho(A_d)\! \mid_{U_k}\}$, $\mathcal{S} := \{ \rho(S_1)\!
\mid_{U_k} , \ldots , \rho(S_r)\! \mid_{U_k}\}$,
 go to (III).
\end{itemize}

\bigskip

At any stage of ${\tt IsFiniteMatGroupFuncFF}$, we test finiteness
of constituents $G|_U$ of $G$ in block triangular form.  In
looping back to step (III) from step (IV), the dimension of the
$G$-module $U$ strictly reduces. Thus, eventually the algorithm
finds either that all constituents are finite, or that one of them
is infinite. In the former case $G$ has a finite homomorphic image
whose kernel is a finitely generated unipotent subgroup of
$\mathrm{GL}(n,\mathbb{F})$, and so is also finite; in the latter
case $G$ is infinite.

The maximum number of iterations of ${\tt IsFiniteMatGroupFuncFF}$
is $2n$, and its main component ${\tt BasisEnvAlgebra}$ has cost
$O(rn^8)$ finite field operations. The principal difference
between ${\tt IsFiniteMatGroupFuncFF}$ and the simpler alternative
${\tt IsFiniteCRMatGroupFuncFF}$ for completely reducible input is
that the former calls ${\tt ModuleViaNullspace}$. The operations
carried out over the function field are matrix addition, matrix
multiplication, and nullspace and intersection of subspaces. All
use $O(n^k)$ field operations where $k \leq 3$. For just one
indeterminate, admissible $\alpha$ always exist in
$\mathbb{F}_{q^{d+1}}$ where $d$ is the largest degree of
denominators in entries of the matrices in $\mathcal{S}$; a
similar estimate holds for $m>1$. In practice, admissible $\alpha$
may be found over a smaller finite field, even the prime subfield.

We turn now to the problem of determining the order of a finite
subgroup of $\mathrm{GL}(n,\mathbb{F})$. Below we give a simple
procedure to solve this problem, based on the next lemma.
\begin{lemma}\label{malcev}
Let $\mathcal{M}$ be a finite subset of
$\mathrm{Mat}(n,\mathbb{F})$. There are infinitely many admissible
$\alpha = (\alpha_1, \ldots , \alpha_m)$, $\alpha_i\in
\overline{\mathbb{F}}_q$, such that $|\mathcal{M}|=
|\mathcal{M}(\alpha)|$. If $m=1$ then $|\mathcal{M}| =
|\mathcal{M}(\alpha)|$ for all but finitely many admissible
$\alpha$.
\end{lemma}
\begin{proof}
Let $\mathcal{M}=\{ M_1, \ldots , M_k\}$. For each pair $i, j$,
where $i< j$, choose a position in which $M_i$ and $M_j$ have
different entries, and let $d_{ij}$ be the difference of the
entries. Denote by $h$ the product $\Pi_{1\leq i < j\leq k}\,
d_{ij}$ of all these differences. If $h(\alpha) \neq 0$ then
$|\mathcal{M}|=|\mathcal{M}(\alpha )|$. Since there are infinitely
many admissible $\alpha$ such that $h(\alpha) \neq 0$, and only
finitely many admissible $\alpha$ such that $h(\alpha) = 0$ if
$m=1$, the result follows.
\end{proof}
\begin{corollary}\label{infinitelymanyalpha}
Let $G\leq \mathrm{GL}(n,\mathbb{F})$ be finite. There are
infinitely many admissible $\alpha$ such that $|G| = |G(\alpha)|$
and $|\langle G\rangle_{\mathbb{F}_{q}}| = |\langle G(\alpha)
\rangle_{\mathbb{F}_{q}}|$. If $m=1$ then $|G| = |G(\alpha)|$ and
$|\langle G\rangle_{\mathbb{F}_{q}}| = |\langle G(\alpha)
\rangle_{\mathbb{F}_{q}}|$ for all but finitely many admissible
$\alpha$.
\end{corollary}
\begin{remark}\label{nogoodalpha}
It is not true that if $G$ is finite then there are infinitely
many admissible $\alpha$ such that $|\langle
G\rangle_{\mathbb{F}_{q^\nu}}| = |\langle G(\alpha)
\rangle_{\mathbb{F}_{q^\nu}}|$. Indeed
$\mathrm{dim}_{\mathbb{F}_{q^\nu}}\langle G(\alpha)
\rangle_{\mathbb{F}_{q^\nu}}$ may be less than
$\mathrm{dim}_{\mathbb{F}_q}\langle G(\alpha)
\rangle_{\mathbb{F}_q}$ for every admissible $\alpha$. For
example, consider the subgroup $G$ of $\mathrm{GL}(2,
\mathbb{F}_2(X))$ generated by $\renewcommand{\arraycolsep}{.1cm}
\tiny{\left( \begin{array}{cc} 1 & 1 \\
0 & 1
\end{array} \right)}$ and
$\renewcommand{\arraycolsep}{.1cm} \tiny{\left( \begin{array}{cc}
1 & X \\ 0 & 1
\end{array} \right)}$.
For all $\alpha \in \overline{\mathbb{F}}_2$ we have
$\mathrm{dim}_{\mathbb{F}_2(\alpha)}\langle G
\rangle_{\mathbb{F}_2(\alpha)} = 3$, whereas
$\mathrm{dim}_{\mathbb{F}_2(\alpha)}\langle G(\alpha)
\rangle_{\mathbb{F}_2(\alpha)} = 2$.
\end{remark}

Corollary~\ref{infinitelymanyalpha} implies that if $G$ is finite
and $m=1$, then there is a positive integer $\delta$ such that
$\varphi_\alpha$ is an isomorphism on $\langle
G\rangle_{\mathbb{F}_{q}}$ whenever $\alpha \in
\overline{\mathbb{F}}_q \setminus \mathbb{F}_{q^\delta}$. As such
$\delta$ may be impracticably large, 
our implementation
of the following algorithm uses the intrinsic random selection function 
in {\sc Magma}.

\bigskip

\hspace*{-1.5mm} ${\tt SizeFiniteMatGroupFuncFF}(\mathcal{S})$

\vspace*{1mm}

Input:  $\mathcal{S} \subseteq \mathrm{GL}(n, \mathbb{F})$ such
that $G = \langle \mathcal{S}\rangle$ is finite.

\vspace*{1mm}

Output: $|G|$.

\vspace*{1mm}

\begin{itemize}

\item[(I)] Randomly select an $\mathcal{S}$-admissible $\alpha \in
\overline{\mathbb{F}}_{q}^{\, (m)}$.

\vspace*{1mm}

\item[(II)] If ${\tt IsIsomorphismEnvAlgebras}(\mathcal{S},\alpha,
1) =$ `true' then return $|G(\alpha)|$;

\vspace*{.3mm}

\noindent else replace $\overline{\mathbb{F}}_{q}^{\, (m)}$ by
$\overline{\mathbb{F}}_{q}^{\, (m)}\setminus \{ \alpha\}$ and go
to (I).
\end{itemize}

\bigskip

We end this section with some comments on ${\tt
SizeFiniteMatGroupFuncFF}$. Recall that
$\mathrm{dim}_{\mathbb{F}_q}\langle G \rangle_{\mathbb{F}_q}$ may
depend exponentially on $n$. However, sometimes we can replace
$(\mathcal{S},\alpha, 1)$ by $(\mathcal{S},\alpha, \nu)$ in step
(II) above, thereby bringing the relevant dimension back to no
more than $n^2$. For instance, this is valid if $G$ is cyclic or
completely reducible. However, in general we cannot make this
modification (\mbox{cf.} Remark~\ref{nogoodalpha}).

Notice that ${\tt SizeFiniteMatGroupFuncFF}$ constructs an
isomorphic copy of $G \leq \mathrm{GL}(n,\mathbb{F})$ defined over
a finite field. We can use this copy and machinery for matrix
groups over finite fields to answer other questions about $G$.

\section{Deciding finiteness of nilpotent matrix groups}
\label{simplenilpotent}

In this section we develop a specialized algorithm to decide
finiteness of nilpotent subgroups of $\mathrm{GL}(n,\mathbb{F})$.
We remove the limitation of \cite[Section 4.3]{Draft} that the
ground field is perfect. Our algorithm represents an improvement
of the positive characteristic finiteness testing algorithm of
\cite{Draft}, including a more efficient transfer to the
completely reducible case. An important application is to decide
whether a single element $g$ of $\mathrm{GL}(n,\mathbb{F})$ 
has finite order.

For the rest of this section, $G\leq \mathrm{GL}(n,\mathbb{F})$ is
nilpotent. We let $g_s$ and $g_u$ denote respectively the
diagonalizable and unipotent parts of $g\in
\mathrm{GL}(n,\mathbb{F})$. Namely, $g_s$ and $g_u$ are the unique
matrices such that $g_s \in \mathrm{GL}(n,\overline{\mathbb{F}})$
is diagonalizable, $g_u \in$
$\mathrm{GL}(n,\overline{\mathbb{F}})$ is unipotent, and $g=
g_sg_u=g_ug_s$.
\begin{lemma}\label{iffiniteingroundfield}
If $g\in \mathrm{GL}(n,\mathbb{F})$ has finite order then $g_s$
and $g_u$ are both in $\langle g \rangle$.
\end{lemma}
\begin{proof}
\mbox{Cf.} \cite[Corollary 1, \mbox{p.} 135]{Segal}.
\end{proof}

Define $G_s=\langle (S_1)_s, \ldots, (S_r)_s\rangle$ and
$G_u=\langle (S_1)_u, \ldots, (S_r)_u\rangle$. The next result
follows from part of \cite[Proposition 3, \mbox{pp.}
136-137]{Segal} (which does not require that the ground field be
perfect).
\begin{lemma}\label{segalsresults} \
\begin{itemize}
\item[{\rm (i)}] The maps defined by $g\mapsto g_s$ and $g\mapsto
g_u$ for $g\in G$ are homomorphisms; thus $G_s =$ $\{ g_s \mid
g\in G\}$ and $G_u = \{ g_u \mid g\in G\}$. \item[{\rm (ii)}]
$G\leq G_s\times G_u$.
\end{itemize}
\end{lemma}
\begin{lemma}\label{finiteiffgsis}
$G$ is finite if and only if $G_s$ is finite.
\end{lemma}
\begin{proof}
By Lemma~\ref{segalsresults} (i), $G_s$ is finite if $G$ is
finite. As a finitely generated periodic matrix group, $G_u$ is
finite. Hence if $G_s$ is finite then $G$ is finite by
Lemma~\ref{segalsresults} (ii).
\end{proof}

Let $\gamma$ be the positive integer such that $p^{\gamma-1}<n\leq
p^\gamma$. By \cite[\mbox{p.} 192]{Suprunenko2}, $p^\gamma$ is the
maximum order of a unipotent element of
$\mathrm{GL}(n,\mathbb{F})$. Define $\mathcal{S}^{p^\gamma}= \{
S_i^{p^\gamma} \mid 1\leq i \leq r \}$ and $G^{p^\gamma} = \langle
\mathcal{S}^{p^\gamma} \rangle$.
\begin{lemma}\label{justifies} \
\begin{itemize}
\item[{\rm (i)}]
 $G$ is
finite if and only if $G^{p^\gamma}$ is finite. \item[{\rm (ii)}]
If $G$ is finite then $G^{p^\gamma}=G_s$ is completely reducible.
\end{itemize}
\end{lemma}
\begin{proof}
(i) Certainly $G^{p^\gamma} \leq G$ is finite if $G$ is finite.
Suppose that $G^{p^\gamma}$ is finite. Then each $S_i$ has finite
order, so $(S_i)_s$ has order coprime to $p$. Thus $(S_i)_s \in
\langle (S_i)_s^{p^\gamma} \rangle$. Since $(S_i)_s^{p^\gamma} \in
\langle S_i^{p^\gamma} \rangle$ by
Lemma~\ref{iffiniteingroundfield}, we have $G_s\leq G^{p^\gamma}$,
and so $G_s$ is finite. Lemma~\ref{finiteiffgsis} now completes
the proof of this item.

(ii) If $G$ is finite then $G_s\leq G^{p^\gamma}$. Further,
$G^{p^\gamma}\leq G_s$ since each generator of the nilpotent group
$G^{p^\gamma} \leq G$ has trivial unipotent part (by the choice of
$\gamma$).
\end{proof}

Lemma~\ref{justifies} justifies correctness of the following.

\bigskip

\hspace*{-1.5mm} ${\tt
IsFiniteNilpotentMatGroupFuncFF}(\mathcal{S})$

\vspace*{1mm}

Input: a finite subset $\mathcal{S}$ of $\mathrm{GL}(n,
\mathbb{F})$ such that $G= \langle \mathcal{S}\rangle$ is
nilpotent.

\vspace*{1mm}

Output: `true' if $G$ is finite; `false' otherwise.

\vspace*{1mm}

\begin{itemize}
\item[(I)] $\mathcal{S}^{p^\gamma}:= \{ S_i^{p^\gamma} \mid 1\leq
i \leq r \}$.

\vspace*{1mm}

\item[(II)] Return ${\tt
IsFiniteCRMatGroupFuncFF}(\mathcal{S}^{p^\gamma})$.
\end{itemize}

\bigskip

For nilpotent input, ${\tt IsFiniteNilpotentMatGroupFuncFF}$ is
superior to ${\tt IsFiniteMatGroup}$-${\tt FuncFF}$, because it
immediately reduces to the completely reducible case.

${\tt IsFiniteNilpotentMatGroupFuncFF}$ may be further refined.
Rather than computing a basis of an enveloping algebra in step
(II), it suffices to test whether $\varphi_\alpha$ has trivial
kernel on $G^{p^\gamma}$. A practical method to do this is given
at the end of \cite[Section 4.2]{Draft}. Likewise, computing
orders can be made more efficient for nilpotent input. A
specialized method to compute the order of a nilpotent subgroup of
$\mathrm{GL}(n,q)$ is implemented in Nilmat \cite{nilmat}, and may
be used in step (II) of ${\tt SizeFiniteMatGroupFuncFF}$.

\section{Implementation and performance}
\label{experiment}

Implementations of our algorithms are publicly available in {\sc
Magma}. In this section we report on their performance and
dependence on the main input parameters: the degree $n$, the
number of generators $r$, and size $q$ of the coefficient field.
We also investigated how runtimes vary with the degrees,
coefficients and number of summands of polynomials appearing in
matrix entries.

The experiments reported in Table~\ref{uniquetable} were
undertaken on a $3.0$ GHz machine with 4GB RAM running {\sc Magma}
V2.15-10.
\begin{table}[htb]
\begin{center}
\caption{}
 \label{uniquetable}
\begin{tabular}{|c|r|r|r|r|r|} \hline
\small{Group} & \small{$n$} & \small{$r$} & \small{$q$} & \small{Runtime.1} & \small{Runtime.2} \\
\hline \hline \small{$G_{11}$} & \small{$40$} & \small{$2$} &
 \small{$5^7$} &  \small{$1646$} & - \\ \hline
\small{$G_{12}$} & \small{$40$} & \small{$10$} &  \small{$5^7$} &
\small{$1124$} & - \\ \hline
\small{$G_{21}$} & \small{$54$} & \small{$20$} & \small{$29^{4}$} & \small{$806$} &  - \\
\hline
\small{$G_{22}$} & \small{$54$} & \small{$23$} & \small{$29^{4}$} & \small{$474$} & - \\
\hline
\small{$G_{31}$} & \small{$36$} & \small{$520$} & \small{$7^8$} &
\small{$2506$} &  \small{$113$}
\\ \hline
\small{$G_{32}$} & \small{$36$} & \small{$522$} & \small{$7^8$} &
\small{$252$} & \small{$20$}
\\ \hline
\small{$G_{41}$} & \small{$100$} & \small{$1$} & \small{$3^{12}$}
& \small{$423$} & \small{$16$}
\\ \hline
\small{$G_{42}$} & \small{$100$} & \small{$1$} & \small{$3^{12}$}
& \small{$8$} & \small{$4$}
\\ \hline
\end{tabular}
\end{center}
\end{table}

As tests, we chose groups with extremal properties, that pass
through all stages of each algorithm. The column `Runtime.1' in
Table~\ref{uniquetable} lists the CPU time in seconds of ${\tt
IsFiniteMatGroupFuncFF}$ for input $G_{ij}$. The column
`Runtime.2' lists the time for ${\tt
IsFiniteNilpotentMatGroupFuncFF}$ when $G_{ij}$ is nilpotent. Note
that the $G_{i1}$ are finite and the $G_{i2}$ are infinite for
$1\leq i \leq 4$.

Polynomials in the matrix entries of $G_{1j}$, $G_{2j}$ have
degrees up to $1000$, and many summands with large coefficients.
The $G_{1j}$ are absolutely irreducible: $G_{11}$ is a conjugate
of $\mathrm{GL}(40,5^7)$ in $\mathrm{GL}(40,\mathbb{F}_{5^7}(X))$,
whereas $G_{12}$ is generated by $G_{11}$ and infinite order
matrices in $\mathrm{SL}(40,\mathbb{F}_{5^7}(X))$. Testing each
group necessitates computing an algebra basis of maximal size
$40^2=1600$ in $\mathrm{Mat}(40,5^7)$. The performance of
${\tt IsFiniteMatGroupFuncFF}$ is essentially identical 
to that of ${\tt IsFiniteCRMatGroupFuncFF}$ for this input.

The $G_{2j}$ have non-trivial unipotent normal subgroups, and so
are not completely reducible. The group $G_{21}$ is the Kronecker
product of a conjugate of $\mathrm{GL}(6,29^{4})$ in
$\mathrm{GL}(6,\mathbb{F}_{29^{4}}(X))$ with a 10-generator
unipotent subgroup of $\mathrm{GL}(9,\mathbb{F}_{29^{4}}(X))$. The
group $G_{22}$ is generated by $G_{21}$ and infinite order
matrices of the form $g\otimes I_9$, where $g$ is an upper
triangular element of $\mathrm{SL}(6,\mathbb{F}_{29^{4}}(X))$.

The $G_{3j}$ are nilpotent and not completely reducible. The group
$G_{31}$ is the Kronecker product of a $3$-dimensional unipotent
group with a $12$-dimensional completely reducible nilpotent group
over $\mathbb{F}_{7^8}(X)$. Specifically, the latter group is a
conjugate of a $2\times 2$ block diagonal group, whose blocks are
a Sylow $3$-subgroup and a Sylow $5$-subgroup of
$\mathrm{SL}(6,7^8)$. The group $G_{32}$ is generated by $G_{31}$
and infinite order diagonal matrices of the form $g\otimes
I_{18}$, where $g\in \mathrm{SL}(2,\mathbb{F}_{7^{8}}(X))$.

The $G_{4j}$ are cyclic. The group $G_{41}$ is generated by $h_1
\otimes \, h$, where $h, h_1\in
\mathrm{GL}(10,\mathbb{F}_{3^{12}}(X))$, $h$ is unipotent, and
$h_1$ is a conjugate of a randomly chosen $3'$-element of
$\mathrm{GL}(10,3^{12})$. Also $G_{42}=$ $\langle h_2\otimes \,
h\rangle$ where $h_2$ is a lower triangular element of
$\mathrm{SL}(10,\mathbb{F}_{3^{12}}(X))$. Comparison of the last
two columns of Table~\ref{uniquetable} for $G_{3j}$ and $G_{4j}$
demonstrates the superiority of ${\tt IsFiniteNilpotentMat}$-${\tt
GroupFuncFF}$ for nilpotent input.

Performance of ${\tt SizeFiniteMatGroupFuncFF}$ depends on the
algorithm used to find the order of a matrix group
over a finite field. {\sc Magma} uses the (random) Schreier-Sims
algorithm \cite[Chapter 7]{HoltEickOBrien05}. In
Table~\ref{secondtable} we report on using ${\tt
SizeFiniteMatGroupFuncFF}$ to compute the orders of the following
groups over a univariate function field: $H_1$ is a conjugate of
the full monomial subgroup of $\mathrm{GL}(20,17)$, $H_2$ and
$H_3$ are nilpotent groups constructed in the same manner as
$G_{31}$ ($H_2$ but not $H_3$ is completely reducible), and $H_4$
is cyclic unipotent.

\begin{table}[htb]
\begin{center}
\caption{}
 \label{secondtable}
\begin{tabular}{|c|r|r|r|r|r|} \hline
\small{Group} & \small{$n$} & \small{$r$} &  \small{$q$} & \small{Order} & \small{Runtime}  \\
\hline \hline \small{$H_{1}$} & \small{$20$} & \small{$3$} &
 \small{$17$} & \small{$20!2^{80}$} & \small{$33$} \\ \hline
\small{$H_{2}$} & \small{$40$} & \small{$24$} & \small{$3^{10}$} &
\small{$5^{22}7^6$} & \small{$56$}
\\ \hline
\small{$H_{3}$} & \small{$24$} & \small{$16$} & \small{$7^2$} &
\small{$3^45^47^3$} & \small{$233$}
\\ \hline
\small{$H_{4}$} & \small{$40$} & \small{$1$} & \small{$5^{10}$} &
\small{$5^3$} & \small{$230$}
\\ \hline
\end{tabular}
\end{center}
\end{table}

\bibliographystyle{amsplain}

\end{document}